\documentclass[a4paper,11pt, reqno]{amsart}
\usepackage{amssymb,amsthm,amsmath}
\usepackage{cite}

\usepackage{amsmath,amsthm,amscd,amssymb}
\usepackage{latexsym}
\usepackage[colorlinks,citecolor=red,pagebackref,hypertexnames=false]{hyperref}
 
\numberwithin{equation}{section}
\allowdisplaybreaks[2]
\theoremstyle{plain}
\newtheorem{theorem}{Theorem}[section]
\newtheorem{Def}[theorem]{Definition}
\newtheorem{lemma}[theorem]{Lemma}
\newtheorem{corollary}[theorem]{Corollary}

\theoremstyle{definition}
\newtheorem{definition}[theorem]{Definition}

\theoremstyle{remark}

\newtheorem{case[theorem]}{Case}

\def \R{{\mathbb R}}

\def \C{{\mathbb C}}

\def\supp{\hbox{supp\,}}
\def\norm#1.#2.{\lVert#1\rVert_{#2}}

\def\R{\mathbb R}

\def \H{{\mathcal H}}

\title[Hausdorff operators associated with the Opdam--Cherednik transform]{Hausdorff operators associated with the Opdam--Cherednik transform in Lebesgue spaces
}

\author{Shyam Swarup Mondal} 
\author{Anirudha Poria}
\thanks{Research supported by ERC Starting Grant No. 713927.}

\address{ \endgraf Department of Mathematics, Indian Institute of Technology Guwahati, Guwahati 781039, India} 
\email{mondalshyam055@gmail.com}

\address{Department of Mathematics, Bar-Ilan University, Ramat-Gan 5290002, Israel}
\email{anirudhamath@gmail.com} 


\date{\today}
\begin{document}
\maketitle
\begin{abstract} 
In this paper, we introduce the Hausdorff operator associated with the Opdam--Cherednik transform and study the boundedness of this operator in various Lebesgue spaces. In particular, we prove the boundedness of the Hausdorff operator in Lebesgue spaces, in grand Lebesgue spaces, and in quasi-Banach spaces that are associated with the Opdam--Cherednik transform.    Also, we give necessary and sufficient conditions for the boundedness of the Hausdorff operator in these spaces. 
\end{abstract} 

\section{Introduction}
One of the most important operators in harmonic analysis is the Hausdorff operator, and it is extremely useful in solving certain classical problems in analysis. This operator originated from some classical summation methods and the Markov moment problem. The Hausdorff operator is deeply rooted in the study of one-dimensional Fourier analysis and has become an essential part of modern harmonic analysis. To study the summability of number series, Hausdorff in \cite{haus} introduced this operator. Then the theory on this operator developed in various directions, for instance, the Hausdorff summability of Fourier series and Hausdorff mean of Fourier--Stieltjes transforms (see \cite{hardy,geor}). To discuss the importance of the Hausdorff operator in more detail, we begin with recalling the definition of this operator.  Let $\phi$ be a locally integrable function on the half-line $(0,\infty)$, then the Hausdorff operator  $H_{\phi}$ on $ \R$ is defined by 
$$
H_{\phi}(f)(x)=\int_{0}^{\infty} \frac{\phi(t)}{t} f\left(\frac{x}{t}\right) \;d t.
$$
By choosing the kernel function $\phi$ appropriately, one can get many classical operators in analysis as a special case of the Hausdorff operator such as the Cesàro operator, Hardy operator, adjoint Hardy operator, Hardy--Littlewood--Pólya operator, Riemann--Liouville fractional integral operator, and many other well-known operators (see \cite{graffa, andersen, christ, lifly09, lifly2019, miyachi}). For a detailed study on the historical development, background, and applications of the Hausdorff operator, we refer to the excellent survey articles of Liflyand  \cite{lif}  and Chen et al. \cite{chen}.

Considerable attention has been devoted to study the basic properties of the Hausdorff operator in various settings. In particular, the boundedness of this operator in different spaces was extensively investigated by many authors. For example, the boundedness of the Hausdorff operator was obtained  in Lebesgue spaces (see \cite{weight, ban20, ban21, chen, 
jain21, liflya19}), in the one-dimensional Hardy space  $H^1(\R)$ (see \cite{ lanker,lif2}), in the Hardy space $H^1(\R^n), n\geq 2$ (see \cite{chen2014, mort, weisez}), and in other function spaces (see \cite{ gao,lif6}). Further, the Hausdorff operator was studied on the Heisenberg group in \cite{fuu}, and weighted Herz space estimates for this operator on the Heisenberg group were obtained in \cite{raun3}. 
Recently, Daher and Saadi in \cite{daher2020,daher2021}   studied the boundedness of the Dunkl--Hausdorff operator in Lebesgue spaces and in the real Hardy space. Motivated by the recent developments of Hausdorff operators and to discovering generalizations for this operator to new contexts, in this paper, we introduce the Hausdorff operator associated with the Opdam--Cherednik transform and study the boundedness of this operator in different Lebesgue spaces.

The motivation to study the Hausdorff operator associated with the Opdam--Cherednik transform in various Lebesgue spaces arises from the Hausdorff operator for the Dunkl transform on function spaces. 
In the setting of  this  transform, we aim to study some basic properties of the Hausdorff operator in Lebesgue spaces.
 The Opdam--Cherednik transform has a significant contribution to harmonic analysis (see \cite{and15,opd95, opd00,sch08}).  An important motivation to study the Jacobi--Cherednik operator arises from their relevance in the algebraic description of exactly solvable quantum many-body systems of Calogero--Moser--Sutherland type (see \cite{die00,hik96}) and they play a crucial role in the study of special functions with root systems (see \cite{dun92,hec91}).  These describe algebraically integrable systems in one dimension and have gained considerable interest in mathematical physics.  A detailed study on the development and applications of the Jacobi--Cherednik operator and Opdam--Cherednik transform can be found in \cite{sch08, ank12,hec91,opd95,trim}. For some recent works on the Opdam--Cherednik transform, we refer to \cite{and15,jahn16, por21, shyam1}.

The purpose of this paper is to study the boundedness of the Hausdorff operator in various Lebesgue spaces associated with the Opdam--Cherednik transform. Mainly, we prove that this operator is bounded in Lebesgue spaces $L^p(\R, A_{\alpha, \beta})$, in grand Lebesgue spaces, and in quasi-Banach spaces. Also, we obtain necessary and sufficient conditions for the boundedness of the Hausdorff operator in these spaces. The motivation and main idea to study the Hausdorff operator in various Lebesgue spaces come from \cite{weight}, where the authors studied the boundedness of the Hausdorff operator in various Lebesgue spaces. Here, we prove that the Hausdorff operator is bounded in various Lebesgue spaces associated with the Opdam--Cherednik transform. The proofs of these results are based on techniques used in \cite{weight}.

The remainder of this paper is structured as follows. In Section \ref{sec2}, we present some preliminaries related to the Opdam--Cherednik transform.  In Section \ref{sec3}, we introduce and study the Hausdorff operator associated with the Opdam--Cherednik transform in different Lebesgue spaces. First, we show that, like in the case of the Fourier transform, the Hausdorff operator satisfies a similar relation for the Opdam--Cherednik transform. Then, we study the boundedness of the Hausdorff operator in Lebesgue spaces $L^p(\R, A_{\alpha, \beta})$, and provide necessary and sufficient conditions for the $L^p(\R, A_{\alpha, \beta})$-boundedness of this operator. Also, we prove  the boundedness of the Hausdorff operator in grand Lebesgue spaces and in quasi-Banach spaces associated with the Opdam--Cherednik transform.  Further, we  give necessary and sufficient  conditions for the boundedness of the Hausdorff operator in these spaces.

\section{Harmonic analysis and the Opdam--Cherednik transform}\label{sec2}
In this section, we recall some necessary definitions and results related to the Opdam--Cherednik transform. For a detailed study on harmonic analysis related to this transform, one can look at  \cite{ opd95, sch08, ank12}.
 Here, we mainly adopt the notation and terminology given in \cite{por21}.

Let $T_{\alpha, \beta}$ denote the Jacobi--Cherednik differential--difference operator (also called the Dunkl--Cherednik operator)
\[T_{\alpha, \beta} f(x)=\frac{d}{dx} f(x)+ \Big[ 
(2\alpha + 1) \coth x + (2\beta + 1) \tanh x \Big] \frac{f(x)-f(-x)}{2} - \rho f(-x), \]
where $\alpha, \beta$ are two parameters satisfying $\alpha \geq \beta \geq -\frac{1}{2}$ and $\alpha > -\frac{1}{2}$, and $\rho= \alpha + \beta + 1$. Let $\lambda \in \C$. The Opdam--Cherednik hypergeometric functions $G^{\alpha, \beta}_\lambda$ on $\R$ are eigenfunctions $T_{\alpha, \beta} G^{\alpha, \beta}_\lambda(x)=i \lambda  G^{\alpha, \beta}_\lambda(x)$ of $T_{\alpha, \beta}$ that are normalized such that $G^{\alpha, \beta}_\lambda(0)=1$. The eigenfunction $G^{\alpha, \beta}_\lambda$ is given by
\[G^{\alpha, \beta}_\lambda (x)= \varphi^{\alpha, \beta}_\lambda (x) - \frac{1}{\rho - i \lambda} \frac{d}{dx}\varphi^{\alpha, \beta}_\lambda (x)=\varphi^{\alpha, \beta}_\lambda (x)+ \frac{\rho+i \lambda}{4(\alpha+1)} \sinh 2x \; \varphi^{\alpha+1, \beta+1}_\lambda (x),  \]
where $\varphi^{\alpha, \beta}_\lambda (x)={}_2F_1 \left(\frac{\rho+i \lambda}{2}, \frac{\rho-i \lambda}{2} ; \alpha+1; -\sinh^2 x \right) $ is the classical Jacobi function.

\noindent For every $ \lambda \in \C$ and $x \in  \R$, the eigenfunction
$G^{\alpha, \beta}_\lambda$ satisfies  
\[ |G^{\alpha, \beta}_\lambda(x)| \leq C \; e^{-\rho |x|} e^{|\text{Im} (\lambda)| |x|},\] 
where $C$ is a positive constant. Since $\rho > 0$, we have
\begin{equation*}
	|G^{\alpha, \beta}_\lambda(x)| \leq C \; e^{|\text{Im} (\lambda)| |x|}. 
\end{equation*}
The Heckman--Opdam hypergeometric functions $F_{\lambda}^{\alpha,\beta}$ satisfy 
$
F_{\lambda}^{\alpha,\beta}(tx)=F_{ \lambda t}^{\alpha,\beta}(x)
,$ for every $x, t \in \R$ (see  \cite{Heckman}).
Since the    Heckman--Opdam   and  Opdam--Cherednik  hypergeometric functions    are related to each other (see\cite{jahn16, sch08,opd95}),   
  the hypergeometric functions $G^{\alpha, \beta}_\lambda$  satisfy the following relation
\begin{align}\label{eq22}
	G^{\alpha, \beta}_{\lambda }(tx)=G^{\alpha, \beta}_{\lambda  t}(x),
\end{align}
 for every $ \lambda \in \C $ and $x, t \in  \R$. For a more detailed study on these hypergeometric functions, we refer to  \cite{Heckman,opd95}.
 
Let us denote by $C_c (\R)$ the space of continuous functions on $\R$ with compact support. The Opdam--Cherednik transform is the Fourier transform in the trigonometric Dunkl setting, and it is defined as follows.
\begin{Def}\label{eq12}
	Let $\alpha \geq \beta \geq -\frac{1}{2}$ with $\alpha > -\frac{1}{2}$. The Opdam--Cherednik transform $\mathcal{H}_{\alpha, \beta} (f)$ of a function $f \in C_c(\R)$ is defined by
	\[ \H_{\alpha, \beta} (f) (\lambda)=\int_{\R} f(x)\; G^{\alpha, \beta}_\lambda(-x)\; A_{\alpha, \beta} (x) dx \quad \text{for all } \lambda \in \C, \] 
	where $A_{\alpha, \beta} (x)= (\sinh |x| )^{2 \alpha+1} (\cosh |x| )^{2 \beta+1}$. The inverse Opdam--Cherednik transform for a suitable function $g$ on $\R$ is given by
	\[ \H_{\alpha, \beta}^{-1} (g) (x)= \int_{\R} g(\lambda)\; G^{\alpha, \beta}_\lambda(x)\; d\sigma_{\alpha, \beta}(\lambda) \quad \text{for all } x \in \R, \]
	where $$d\sigma_{\alpha, \beta}(\lambda)= \left(1- \dfrac{\rho}{i \lambda} \right) \dfrac{d \lambda}{8 \pi |C_{\alpha, \beta}(\lambda)|^2}$$ and 
	$$C_{\alpha, \beta}(\lambda)= \dfrac{2^{\rho - i \lambda} \Gamma(\alpha+1) \Gamma(i \lambda)}{\Gamma \left(\frac{\rho + i \lambda}{2}\right)\; \Gamma\left(\frac{\alpha - \beta+1+i \lambda}{2}\right)}, \quad \lambda \in \C \setminus i \mathbb{N}.$$
\end{Def}

The Plancherel formula is given by 
\begin{equation}\label{eq03}
	\int_{\R} |f(x)|^2 A_{\alpha, \beta}(x) dx=\int_\R \H_{\alpha, \beta} (f)(\lambda) \overline{\H_{\alpha, \beta} ( \check{f})(-\lambda)} \; d \sigma_{\alpha, \beta} (\lambda),
\end{equation}
where $\check{f}(x):=f(-x)$.

Let $L^p(\R,A_{\alpha, \beta} )$ (resp. $L^p(\R, \sigma_{\alpha, \beta} )$), $p \in [1, \infty] $, denote the $L^p$-spaces corresponding to the measure $A_{\alpha, \beta}(x) dx$ (resp. $d | \sigma_{\alpha, \beta} |(x)$). We refer to \cite{ank12, hec91, opd95, trim,  opd00} for further properties and results related to  the Opdam--Cherednik transform.
 
\section{Main results}\label{sec3}
In this section, we define the Hausdorff operator associated with the Opdam--Cherednik transform, and study the boundedness of this operator in different Lebesgue spaces. Here, we consider various Lebesgue spaces associated with the Opdam--Cherednik transform. We begin with the definition of the Hausdorff operator.
\begin{definition}
	Let $\phi$ be a non-negative function  defined on $(0,\infty)$ and $\phi \in L^{1} (0, \infty)$, then the Hausdorff operator $H_{\alpha, \beta, \phi}$ acting on $L^1(\R, A_{\alpha, \beta})$  generated by the function $\phi$, is defined by
	\begin{align}\label{eq11}
		H_{\alpha, \beta, \phi}(f)(x)=\int_{0}^{\infty} \frac{\phi(t)}{t} f\left(\frac{x}{t}\right) \frac{A_{\alpha, \beta}(\frac{x}{t})}{A_{\alpha, \beta}(x)}\; dt, \quad x\in \R.
	\end{align}
\end{definition}

Next, we provide some examples of the Hausdorff operator associated with the Opdam--Cherednik transform. By choosing the function $\phi$ appropriately, we can get many classical operators associated with the Opdam--Cherednik transform as special cases of the Hausdorff operator. For example:

\begin{enumerate}

\item  if $\phi(t)=\frac{\chi_{(1, \infty)}(t)}{t}$, we obtain the Hardy operator associated with the Opdam--Cherednik transform
\[ H f(x)=H_{\alpha, \beta, \phi} f(x)=\frac{1}{x} \int_{0}^{x} f(t) \; \frac{A_{\alpha, \beta}(t)}{A_{\alpha, \beta}(x)} \; dt ; \]

\item  if $\phi(t)=\chi_{(0,1)}(t)$, we get the adjoint Hardy operator associated with the Opdam--Cherednik transform
\[ H^{*} f(x) = H_{\alpha, \beta, \phi} f(x) = \int_{x}^{\infty} \frac{f(t)}{t} \; \frac{A_{\alpha, \beta}(t)}{A_{\alpha, \beta}(x)} \; dt ;\] 

\item if $\phi(t)=\frac{1}{\max \{1, t\}}$, we have the Hardy--Littlewood--Pólya operator associated with the Opdam--Cherednik transform
\[ Pf(x) = H_{\alpha, \beta, \phi} f (x)=\frac{1}{x} \int_{0}^{x} f(t) \; \frac{A_{\alpha, \beta}(t)}{A_{\alpha, \beta}(x)} \;dt + \int_{x}^{\infty} \frac{f(t)}{t} \; \frac{A_{\alpha, \beta}(t)}{A_{\alpha, \beta}(x)}\; dt; \]

\item if $\phi(t)=\gamma(1-t)^{\gamma-1} \chi_{(0,1)}(t)$ with $\gamma>0$, we get the Cesàro operator associated with the Opdam--Cherednik transform
\[ \mathcal{C}_{\gamma} f(x)=H_{\alpha, \beta, \phi} f(x)=\gamma \int_{x}^{\infty} \frac{(t-x)^{\gamma-1}}{t^\gamma} \; f(t) \; \frac{A_{\alpha, \beta}(t)}{A_{\alpha, \beta}(x)} \; dt ; \]

\item if $\phi(t)=\frac{1}{\Gamma(\beta)} \frac{\left(1-\frac{1}{t}\right)^{\beta-1}}{t} \chi_{(1, \infty)}(t)$ with $\beta>0$, we obtain the Riemann--Liouville fractional derivative associated with the Opdam--Cherednik transform in the following form
\[ D_{\beta} f(x)=x^{\beta} H_{\alpha, \beta, \phi} f(x)=\frac{1}{\Gamma(\beta)} \int_{0}^{x}(x-t)^{\beta-1} f(t)\; \frac{A_{\alpha, \beta}(t)}{A_{\alpha, \beta}(x)}  \; dt.\]

\end{enumerate}

\subsection{Boundedness of the Hausdorff operator in Lebesgue spaces}In this subsection, we study the boundedness of the Hausdorff operator in Lebesgue spaces associated with the Opdam--Cherednik transform. 
First,  we show that the operator $H_{\alpha, \beta, \phi}$  is bounded on $L^1(\R, A_{\alpha, \beta}).$
\begin{theorem}
	Let $\phi \in L^{1} (0, \infty)$. Then   $H_{\alpha, \beta, \phi}: L^1(\R, A_{\alpha, \beta})  \rightarrow L^1(\R, A_{\alpha, \beta})$ is a bounded operator and 
	$$
	\left\|H_{\alpha, \beta, \phi} f\right\|_{L^1(\R, A_{\alpha, \beta})} \leq\|\phi\|_{L^1(0,\infty)} \;\|f\|_{L^1(\R, A_{\alpha, \beta})},
	$$
	for   $f \in L^1(\R, A_{\alpha, \beta}).$ 
\end{theorem}
\begin{proof}
	For any $f \in L^1(\R, A_{\alpha, \beta})$, using Fubini’s theorem, we get  
	\begin{align*}
			\left\|H_{\alpha, \beta, \phi} f\right\|_{L^1(\R, A_{\alpha, \beta})}
			&=\int_{\R} 	|H_{\alpha, \beta, \phi} f(x)|\; A_{\alpha, \beta}(x)dx\\
			&=\int_{\R} 	\left |\int_{0}^{\infty} \frac{\phi(t)}{t} f\left(\frac{x}{t}\right) \frac{A_{\alpha, \beta}(\frac{x}{t})}{A_{\alpha, \beta}(x)}\; dt\right| A_{\alpha, \beta}(x)dx\\
						&\leq  \int_{0}^{\infty} \frac{\phi(t)}{t} \left(  \int_{\R}	\left |f\left(\frac{x}{t}\right)\right| A_{\alpha, \beta}\left(\frac{x}{t}\right) \; dx\right) dt.
	\end{align*} 
Using the   change of variable 
 $x 
\mapsto  u = \frac{x}{t}$ in the second integral,  we obtain  
	\begin{align*}
	\left\|H_{\alpha, \beta, \phi} f\right\|_{L^1(\R, A_{\alpha, \beta})}
&\leq  \int_{0}^{\infty} {\phi(t)} \left(   \int_{\R}	\left |f(u)\right| A_{\alpha, \beta}(u) \; du\right) dt=\|\phi\|_{{L^1(0,\infty)} } \;\|f\|_{L^1(\R, A_{\alpha, \beta})}.
\end{align*} 
This completes the proof.
\end{proof}
In the following theorem, we show that, like in the case of the Fourier transform, the Hausdorff operator defined in (\ref{eq11}) satisfies the similar relation  for the Opdam--Cherednik transform. 
\begin{theorem}
Let $\phi \in L^{1} (0, \infty)$.  Then  for any 
$f \in L^1(\R, A_{\alpha, \beta}),$ the Opdam--Cherednik transform  $\H_{\alpha, \beta}$ of  
$H_{\alpha, \beta, \phi}f $ satisfies
	$$
\H_{\alpha, \beta} \left(H_{\alpha, \beta, \phi}f \right)(\lambda)=\int_{0}^{\infty}  \H_{\alpha, \beta}  (f)(\lambda  t) \phi(t) \;dt, \quad \lambda \in \mathbb{R}.
	$$
\end{theorem}
\begin{proof}
		For any $f \in L^1(\R, A_{\alpha, \beta})$, using   Definition \ref{eq12} and Fubini's theorem, we get 
	\begin{align*}
		 \H_{\alpha, \beta} (H_{\alpha, \beta, \phi}f ) (\lambda)&=\int_{\R} H_{\alpha, \beta, \phi}f (x)\; G^{\alpha, \beta}_\lambda(-x)\; A_{\alpha, \beta} (x) dx\\
		 &=\int_{\R} \left(\int_{0}^{\infty} \frac{\phi(t)}{t} f\left(\frac{x}{t}\right) \frac{A_{\alpha, \beta}(\frac{x}{t})}{A_{\alpha, \beta}(x)}\; dt\right)G^{\alpha, \beta}_\lambda(-x)\; A_{\alpha, \beta} (x) dx\\
		 		 &=\int_{0}^{\infty} \frac{\phi(t)}{t} \left( \int_{\R}  f\left(\frac{x}{t}\right)G^{\alpha, \beta}_\lambda(-x) A_{\alpha, \beta}\left (\frac{x}{t}\right)\; dx\right) dt.
	\end{align*}
Using the   change of variable 
$x 
\mapsto  u = \frac{x}{t}$ in the second integral and the relation (\ref{eq22}), we obtain 
	\begin{align*}
	\H_{\alpha, \beta} (H_{\alpha, \beta, \phi}f ) (\lambda)&=\int_{0}^{\infty} {\phi(t)} \left( \int_{\R}  f(u) G^{\alpha, \beta}_\lambda(-ut) A_{\alpha, \beta}(u)\; du\right) dt\\
	&=\int_{0}^{\infty} {\phi(t)} \left( \int_{\R}  f(u) G^{\alpha, \beta}_{\lambda  t}(-u) A_{\alpha, \beta}(u)\; du\right) dt\\
		&=\int_{0}^{\infty}   \H_{\alpha, \beta}  (f)(\lambda  t) \phi(t)\;dt.
\end{align*}
Since $|G^{\alpha, \beta}_\lambda(-x) |\leq 1$, the absolute convergence of these double integrals justifies the above calculations.
\end{proof}
We define  two quantities $	A_{\sup} $ and $A_{\inf }$   as
$$
\begin{aligned}
	A_{\sup} &:= \int_{0}^{\infty} \frac{\phi(t)}{t} t^{\frac{1}{p}}\left(\sup_{u\in \R}\frac{A_{\alpha, \beta}(u)}{A_{\alpha, \beta}(tu)} \right)^{1-\frac{1}{p}}  dt,  \\
	A_{\inf } &:= \int_{0}^{\infty} \frac{\phi(t)}{t} t^{\frac{1}{p}}\left(\inf_{u\in \R}\frac{A_{\alpha, \beta}(u)}{A_{\alpha, \beta}(tu)} \right)^{1-\frac{1}{p}}  dt.
\end{aligned}
$$
Next,  we prove   the  boundedness of the Hausdorff operator  in  $ L^p(\R, A_{\alpha, \beta})$. 
\begin{theorem}\label{eq6}
	Let $1<p<\infty $  and   $\phi \in L ^1(0, \infty ).$  If $A_{\text {sup }}<\infty$, then    $H_{\alpha, \beta, \phi}: L^p(\R, A_{\alpha, \beta})  \rightarrow L^p(\R, A_{\alpha, \beta})$ is a bounded operator with 
	$$
	\left\|H_{\alpha, \beta, \phi} f\right\|_{L^p(\R, A_{\alpha, \beta})} \leq A_{\sup}\;\|f\|_{L^p (\R, A_{\alpha, \beta})},
	$$  for    $f \in L^p(\R, A_{\alpha, \beta})$.
\end{theorem}
\begin{proof}
	For any $f \in L^p(\R, A_{\alpha, \beta})$, using the generalized Minkowski inequality, we get  
	\begin{align*}
		\left\|H_{\alpha, \beta, \phi} f\right\|_{L^p(\R, A_{\alpha, \beta})}
		&=\left( \int_{\R} 	|H_{\alpha, \beta, \phi} f(x)|^p\; A_{\alpha, \beta}(x)dx\right)^{\frac{1}{p}}\\
		&=\left( \int_{\R} 	\left |\int_{0}^{\infty} \frac{\phi(t)}{t} f\left(\frac{x}{t}\right) \frac{A_{\alpha, \beta}(\frac{x}{t})}{A_{\alpha, \beta}(x)}\; dt\right|^p A_{\alpha, \beta}(x)dx\right)^{\frac{1}{p}}\\
		&\leq  \int_{0}^{\infty} \frac{\phi(t)}{t} \left(  \int_{\R}	\left |f\left(\frac{x}{t}\right)\right|^p \left(\frac{A_{\alpha, \beta}(\frac{x}{t})}{A_{\alpha, \beta}(x)} \right)^p A_{\alpha, \beta}(x) dx\right)^{\frac{1}{p}} dt.
	\end{align*} 
Using the   change of variable  $x  \mapsto  u = \frac{x}{t}$ in the second integral,  we obtain  
\begin{align*}
	\left\|H_{\alpha, \beta, \phi} f\right\|_{L^p(\R, A_{\alpha, \beta})}
	&\leq  \int_{0}^{\infty} \frac{\phi(t)}{t} t^{\frac{1}{p}}  \left( \int_{\R}	 |f(u)|^p \left(\frac{A_{\alpha, \beta}(u)}{A_{\alpha, \beta}(tu)} \right)^p A_{\alpha, \beta}(tu) \; du\right)^{\frac{1}{p}} dt\\
	&\leq  \int_{0}^{\infty} \frac{\phi(t)}{t} t^{\frac{1}{p}}\left(\sup_{u\in \R}\frac{A_{\alpha, \beta}(u)}{A_{\alpha, \beta}(tu)} \right)^{1-\frac{1}{p}}  \left(  \int_{\R}	 |f(u)|^p  A_{\alpha, \beta}(u) \; du \right)^{\frac{1}{p}} dt\\
		&= \left(  \int_{0}^{\infty} \frac{\phi(t)}{t} t^{\frac{1}{p}}\left(\sup_{u\in \R}\frac{A_{\alpha, \beta}(u)}{A_{\alpha, \beta}(tu)} \right)^{1-\frac{1}{p}}  dt \right) \|f\|_{L^p (\R, A_{\alpha, \beta})}\\
		&=A_{\sup}\;\|f\|_{L^p (\R, A_{\alpha, \beta})},
\end{align*} 
which completes the proof.
\end{proof}
In the following theorem, we  give  a necessary condition for the  $ L^p(\R, A_{\alpha, \beta})$-boundedness of  the operator  $H_{\alpha, \beta, \phi} $.
\begin{theorem}\label{eq4}
Let $1<p<\infty $ and $A_{\inf} > 0$. If  $H_{\alpha, \beta, \phi}: L^p(\R, A_{\alpha, \beta})  \rightarrow L^p(\R, A_{\alpha, \beta})$ is a bounded operator, then  
$$
\left\|H_{\alpha, \beta, \phi} \right\|_{L^p (\R, A_{\alpha, \beta})\to L^p (\R, A_{\alpha, \beta})} \geq A_{\inf}.$$
\end{theorem}
\begin{proof}
Assume that    $H_{\alpha, \beta, \phi}: L^p(\R, A_{\alpha, \beta})  \rightarrow L^p(\R, A_{\alpha, \beta})$ is a bounded operator. For  $0<\varepsilon< 1$ fixed, we  consider the  function 
$$
f_\varepsilon(x)= x^{-\frac{1}{p}-\varepsilon} A_{\alpha, \beta}(x)^{-\frac{1}{p}}\chi_{(1, \infty )}(x).
$$
Then 
$$\|f_\varepsilon\|_{L^p (\R, A_{\alpha, \beta})}=\left( \int_{\R} 	|f_\varepsilon(x)|^p\; A_{\alpha, \beta}(x)dx\right)^{\frac{1}{p}}=\left( \int_{1}^\infty 	x^{-1- \varepsilon p}\; dx\right)^{\frac{1}{p}}=\frac{1}{(\varepsilon p)^{\frac{1}{p}}}.$$
Also,  we have
\begin{align*}
	H_{\alpha, \beta, \phi}f_\varepsilon(x)=\int_{0}^{\infty} \frac{\phi(t)}{t} f_\varepsilon\left(\frac{x}{t}\right) \frac{A_{\alpha, \beta}(\frac{x}{t})}{A_{\alpha, \beta}(x)}\; dt=x^{-\frac{1}{p}-\varepsilon}\int_{0}^{x} \frac{\phi(t)}{t} t^{\frac{1}{p}+\varepsilon}  \frac{A_{\alpha, \beta}(\frac{x}{t})^{1-\frac{1}{p}}}{A_{\alpha, \beta}(x)}\; dt.
\end{align*}
Therefore,
	\begin{align*} 
	\left\|H_{\alpha, \beta, \phi} f_\varepsilon\right\|_{L^p(\R, A_{\alpha, \beta})}
	&=\left( \int_{\R} 	|H_{\alpha, \beta, \phi} f_\varepsilon(x)|^p\; A_{\alpha, \beta}(x)dx\right)^{\frac{1}{p}}\\\nonumber
	&\geq \left( \int_{\frac{1}{\varepsilon}}^\infty 	x^{-1- \varepsilon p} \left ( \int_{0}^{x} \frac{\phi(t)}{t}  t^{\frac{1}{p}+\varepsilon} \frac{A_{\alpha, \beta}(\frac{x}{t})^{1-\frac{1}{p}}}{A_{\alpha, \beta}(x)}\; dt \right)^p A_{\alpha, \beta}(x)dx\right)^{\frac{1}{p}}\\\nonumber
		&\geq \left( \int_{\frac{1}{\varepsilon}}^\infty 	x^{-1- \varepsilon p} \left ( \int_{0}^{\frac{1}{\varepsilon}}\frac{\phi(t)}{t} t^{\frac{1}{p}+\varepsilon}  \left( \frac{A_{\alpha, \beta}(\frac{x}{t})}{A_{\alpha, \beta}(x)}\right)^{1-\frac{1}{p}} \; dt \right)^p  dx\right)^{\frac{1}{p}}\\
&\geq   \frac{\varepsilon^\varepsilon}{(\varepsilon p)^{\frac{1}{p}}} \int_{0}^{\frac{1}{\varepsilon}} \frac{\phi(t)}{t} t^{\frac{1}{p}+\varepsilon} \left(\inf_{x\in \R} \frac{A_{\alpha, \beta}(\frac{x}{t})}{A_{\alpha, \beta}(x)}\right)^{1-\frac{1}{p}} \; dt\\
&= \frac{\varepsilon^\varepsilon}{(\varepsilon p)^{\frac{1}{p}}} \int_{0}^{\frac{1}{\varepsilon}}\frac{\phi(t)}{t} t^{\frac{1}{p}+\varepsilon}   \left(\inf_{x\in \R} \frac{A_{\alpha, \beta}(x)}{A_{\alpha, \beta}(tx)}\right)^{1-\frac{1}{p}} \; dt.
\end{align*} 
Thus,
\begin{align}\label{eq2}
\left\|H_{\alpha, \beta, \phi} \right\|_{L^p (\R, A_{\alpha, \beta})\to L^p (\R, A_{\alpha, \beta})} \geq \varepsilon^\varepsilon \int_{0}^{\frac{1}{\varepsilon}} \frac{\phi(t)}{t} t^{\frac{1}{p}+\varepsilon}  \left(\inf_{x\in \R} \frac{A_{\alpha, \beta}(x)}{A_{\alpha, \beta}(tx)}\right)^{1-\frac{1}{p}} \; dt.
\end{align}
Finally, applying  the Fatou lemma and taking the limit  $\varepsilon \to 0$, we obtain that    the right hand side     of  (\ref{eq2}) converges to $A_{\inf }$ and  this completes the proof of the theorem.
\end{proof}
In the following, we give a  characterization for the boundedness of the  Hausdorff operator $H_{\alpha, \beta, \phi}: L^p(\R, A_{\alpha, \beta}) \rightarrow L^p(\R, A_{\alpha, \beta})$   using Theorems \ref{eq6} and \ref{eq4}.
\begin{corollary}
	Let $1<p<\infty $ and 
	$$
\sup_{u\in \R}\frac{A_{\alpha, \beta}(u)}{A_{\alpha, \beta}(tu)} \leq  C \inf_{u\in \R}\frac{A_{\alpha, \beta}(u)}{A_{\alpha, \beta}(tu)},
	$$
	for some positive constant $C .$ Then,  the operator $H_{\alpha, \beta, \phi}: L^p(\R, A_{\alpha, \beta})  \rightarrow L^p(\R, A_{\alpha, \beta})$ is  bounded    if and only if 
$0 < A_{\sup} <\infty.$  Also,  the
following estimates hold
	$$
	\frac{1}{C^{1-\frac{1}{p}}} A_{\sup} \leq \left\|H_{\alpha, \beta, \phi} \right\|_{L^p (\R, A_{\alpha, \beta})\to L^p (\R, A_{\alpha, \beta})} \leq  A_{\sup}.
	$$
\end{corollary}
Next, we   obtain a  sufficient  condition for the  boundedness  of  the    operator $H_{\alpha, \beta, \phi} :L^p(\R, A_{\alpha, \beta})\to L^q(\R, A_{\alpha, \beta})$.
\begin{theorem} 
	Let $1<q<p<\infty $  and   $\phi \in L ^1(0, \infty )$   be  such that 
	 $$C=\int_{0}^{\infty} \frac{\phi(t)}{t} t^{\frac{1}{q}}  \left( \int_{\R}	 \left( \frac{A_{\alpha, \beta}(u)^{q-\frac{q}{p}}}{A_{\alpha, \beta}(tu)^{q-1}} \right)^{\frac{p}{p-q}}  \; du\right)^{\frac{p-q}{pq}} dt<\infty.$$ 
	Then   $H_{\alpha, \beta, \phi}: L^p(\R, A_{\alpha, \beta})  \rightarrow L^q(\R, A_{\alpha, \beta})$ is a bounded operator and 
	$$
	\left\|H_{\alpha, \beta, \phi} f\right\|_{L^q(\R, A_{\alpha, \beta})} \leq C \;\|f\|_{L^p (\R, A_{\alpha, \beta})}.
	$$
\end{theorem}
\begin{proof}
	For every $f \in L^p(\R, A_{\alpha, \beta})$, using the generalized Minkowski inequality, we get  
	\begin{align*}
		\left\|H_{\alpha, \beta, \phi} f\right\|_{L^q(\R, A_{\alpha, \beta})}
		&=\left( \int_{\R} 	|H_{\alpha, \beta, \phi} f(x)|^q\; A_{\alpha, \beta}(x)dx\right)^{\frac{1}{q}}\\
		&=\left( \int_{\R} 	\left |\int_{0}^{\infty} \frac{\phi(t)}{t} f\left(\frac{x}{t}\right) \frac{A_{\alpha, \beta}(\frac{x}{t})}{A_{\alpha, \beta}(x)}\; dt\right|^q A_{\alpha, \beta}(x)dx\right)^{\frac{1}{q}}\\
		&\leq  \int_{0}^{\infty} \frac{\phi(t)}{t} \left(  \int_{\R}	\left |f\left(\frac{x}{t}\right)\right|^q \left(\frac{A_{\alpha, \beta}(\frac{x}{t})}{A_{\alpha, \beta}(x)} \right)^q A_{\alpha, \beta}(x) \; dx\right)^{\frac{1}{q}} dt.
	\end{align*} 
Using the   change of variable  $x  \mapsto  u = \frac{x}{t}$ in the second integral  and H\"older's inequality, we obtain
	\begin{align*}
		\left\|H_{\alpha, \beta, \phi} f\right\|_{L^q(\R, A_{\alpha, \beta})}
		&\leq  \int_{0}^{\infty} \frac{\phi(t)}{t} t^{\frac{1}{q}}  \left(  \int_{\R}	 |f(u)|^q \left(\frac{A_{\alpha, \beta}(u)}{A_{\alpha, \beta}(tu)} \right)^q A_{\alpha, \beta}(tu) \; du\right)^{\frac{1}{q}} dt\\
		&=  \int_{0}^{\infty} \frac{\phi(t)}{t} t^{\frac{1}{q}} \left(  \int_{\R}	 |f(u)|^q A_{\alpha, \beta}(u)^{\frac{q}{p}}  \frac{A_{\alpha, \beta}(u)^{q-\frac{q}{p}}}{A_{\alpha, \beta}(tu)^{q-1}}   \; du\right)^{\frac{1}{q}} dt\\
		&\leq  \left( \int_{0}^{\infty} \frac{\phi(t)}{t} t^{\frac{1}{q}}  \left(  \int_{\R}	 \left( \frac{A_{\alpha, \beta}(u)^{q-\frac{q}{p}}}{A_{\alpha, \beta}(tu)^{q-1}} \right)^{\frac{p}{p-q}}  \; du\right)^{\frac{p-q}{pq}} dt\right) \|f\|_{L^p (\R, A_{\alpha, \beta})}\\
		&=C\|f\|_{L^p (\R, A_{\alpha, \beta})}.
	\end{align*} 
	This   completes the proof.
\end{proof}
In the following, we obtain   the boundedness of the Hausdorff operator   in  $L^p((0, 1), A_{\alpha, \beta})$.
\begin{theorem}\label{eq13}
		Let $1<p<\infty $  and   $\phi \in L ^1(0, \infty )$ be  such that $\supp \phi \subset [1, \infty).$ Then, the following conditions are equivalent
		\begin{enumerate}
			\item $E(\phi, p)=\int_{1}^{\infty} \frac{\phi(t)}{t} t^{\frac{1}{p}} \;dt<\infty$,\\
			\item  $H_{\alpha, \beta, \phi}: L^p((0, 1), A_{\alpha, \beta})  \rightarrow L^p((0, 1), A_{\alpha, \beta})$ is a bounded operator.
		\end{enumerate}		
\end{theorem}
\begin{proof}
 First, assume that  $E(\phi, p)<\infty$. Then, for every $f \in L^p((0, 1), A_{\alpha, \beta})$, using the generalized Minkowski inequality, we get  
\begin{align*}
	\left\|H_{\alpha, \beta, \phi} f\right\|_{L^p((0, 1), A_{\alpha, \beta})}
	&=\left( \int_{0}^1 	|H_{\alpha, \beta, \phi} f(x)|^p\; A_{\alpha, \beta}(x)dx\right)^{\frac{1}{p}}\\
	&=\left( \int_{0}^1 	\left |\int_{0}^{\infty} \frac{\phi(t)}{t} f\left(\frac{x}{t}\right) \frac{A_{\alpha, \beta}(\frac{x}{t})}{A_{\alpha, \beta}(x)}\; dt\right|^p A_{\alpha, \beta}(x)dx\right)^{\frac{1}{p}}\\
	&\leq  \int_{0}^{\infty} \frac{\phi(t)}{t} \left( \int_{0}^1	\left |f\left(\frac{x}{t}\right)\right|^p \left(\frac{A_{\alpha, \beta}(\frac{x}{t})}{A_{\alpha, \beta}(x)} \right)^p A_{\alpha, \beta}(x) \; dx\right)^{\frac{1}{p}} dt.
\end{align*} 
Using the   change of variable  $x  \mapsto  u = \frac{x}{t}$ in the second integral, we obtain 
\begin{align*}
	\left\|H_{\alpha, \beta, \phi} f\right\|_{L^p((0, 1), A_{\alpha, \beta})}
	&\leq  \int_{1}^{\infty} \frac{\phi(t)}{t} t^{\frac{1}{p}}  \left( \int_{0}^{\frac{1}{t}}	 |f(u)|^p \left(\frac{A_{\alpha, \beta}(u)}{A_{\alpha, \beta}(tu)} \right)^p A_{\alpha, \beta}(tu) \; du\right)^{\frac{1}{p}} dt\\
		&\leq  \int_{1}^{\infty} \frac{\phi(t)}{t} t^{\frac{1}{p}}  \left( \int_{0}^{1}	 |f(u)|^p \left(\frac{A_{\alpha, \beta}(u)}{A_{\alpha, \beta}(tu)} \right)^p A_{\alpha, \beta}(tu) \; du\right)^{\frac{1}{p}} dt\\
	&\leq A_{\alpha, \beta}(1)^{1-\frac{1}{p}} \left( \int_{1}^{\infty} \frac{\phi(t)}{t} t^{\frac{1}{p}}\;dt\right)  \left( \int_{0}^1	 |f(u)|^p  A_{\alpha, \beta}(u) \; du\right) ^{\frac{1}{p}} \\
	&= A_{\alpha, \beta}(1)^{1-\frac{1}{p}} E(\phi, p) 	\left\|  f\right\|_{L^p((0, 1), A_{\alpha, \beta})}.
\end{align*} 
This shows that  $H_{\alpha, \beta, \phi}: L^p((0, 1), A_{\alpha, \beta})  \rightarrow L^p((0, 1), A_{\alpha, \beta})$ is a bounded operator.

 Next,   assume that  the operator $H_{\alpha, \beta, \phi}: L^p((0, 1), A_{\alpha, \beta})  \rightarrow L^p((0, 1), A_{\alpha, \beta})$ is bounded.  For a fixed $\delta$ with  $0<\delta< \frac{1}{p}$, we define  the  function 
$$
f_\delta(x)= x^{\delta-\frac{1}{p} } A_{\alpha, \beta}(x)^{-\frac{1}{p}}, \quad x\in  (0, 1).
$$
Then, we have 
$$\|f_\delta\|_{L^p ((0,1), A_{\alpha, \beta})}=\left( \int_{0}^1 	|f_\delta(x)|^p\; A_{\alpha, \beta}(x)dx\right)^{\frac{1}{p}}=\left( \int_{0}^1 	x^{\delta p-1}\; dx\right)^{\frac{1}{p}}=\frac{1}{(\delta p)^{\frac{1}{p}}}.$$
Moreover,  for any $x\in (0, 1)$, we get 
\begin{align}\label{eq14}\nonumber
	H_{\alpha, \beta, \phi}f_\delta(x)&=\int_{0}^{\infty} \frac{\phi(t)}{t} f_\delta\left(\frac{x}{t}\right) \frac{A_{\alpha, \beta}(\frac{x}{t})}{A_{\alpha, \beta}(x)}\; dt\\\nonumber
	&=x^{\delta-\frac{1}{p}} A_{\alpha, \beta}(x)^{-\frac{1}{p}}\int_{1}^{\infty}\frac{\phi(t)}{t} t^{\frac{1}{p}-\delta}  \frac{A_{\alpha, \beta}(\frac{x}{t})^{1-\frac{1}{p}}}{A_{\alpha, \beta}(x)^{1-\frac{1}{p}}}\; dt\\\nonumber
	&\geq \frac{1}{A_{\alpha, \beta}(1)^{1-\frac{1}{p}}} x^{\delta-\frac{1}{p}}A_{\alpha, \beta}(x)^{-\frac{1}{p}}\int_{1}^{\infty}\frac{\phi(t)}{t}t^{\frac{1}{p}-\delta}  \; dt\\
		&= \frac{1}{A_{\alpha, \beta}(1)^{1-\frac{1}{p}}} E\left(\phi, \frac{p}{1-\delta p}\right)f_\delta(x) .
\end{align}
Therefore,
\begin{align*} 
	\left\|H_{\alpha, \beta, \phi} f_\delta\right\|_{L^p((0, 1), A_{\alpha, \beta})}
	&\geq \frac{1}{A_{\alpha, \beta}(1)^{1-\frac{1}{p}}}  E\left(\phi, \frac{p}{1-\delta p}\right)\|f_\delta\|_{L^p((0, 1), A_{\alpha, \beta})},
\end{align*} 
 and thus
\begin{align*} 
	\left\|H_{\alpha, \beta, \phi} \right\|_{L^p ((0, 1), A_{\alpha, \beta})\to L^p ((0, 1), A_{\alpha, \beta})} \geq  \frac{1}{A_{\alpha, \beta}(1)^{1-\frac{1}{p}}}  E\left(\phi, \frac{p}{1-\delta p}\right).
\end{align*}
Now, taking the limit   $\delta \to 0$, we obtain 
\begin{align*} 
	\left\|H_{\alpha, \beta, \phi} \right\|_{L^p ((0, 1), A_{\alpha, \beta})\to L^p ((0, 1), A_{\alpha, \beta})} \geq  \frac{1}{A_{\alpha, \beta}(1)^{1-\frac{1}{p}}}  E\left(\phi, p\right),
\end{align*}
and this completes the proof of the theorem.
\end{proof}
\subsection{Boundedness of the Hausdorff operator in grand Lebesgue spaces}
Let $I \subset (0, \infty)$  be  such that $A_{\alpha, \beta}(I)<\infty.$ Then, the  {\it grand Lebesgue space}  $L^{p)}(I, A_{\alpha, \beta})$ associated with the Opdam--Cherednik transform is  the class of all measurable functions $f: I \rightarrow \R$ such that 
$$
\|f\|_{L^{p)}(I, A_{\alpha, \beta})}:=\sup _{0<\varepsilon<p-1} \varepsilon^{\frac{1}{p-\varepsilon}}\left(\frac{1}{A_{\alpha, \beta}(I)} \int_{I}|f(x)|^{p-\varepsilon} \;A_{\alpha, \beta} (x)d x\right)^{\frac{1}{p-\varepsilon}}<\infty.
$$
The  grand Lebesgue space   was introduced by Iwaniec and Sbordone in \cite{iwan}. For a more detailed  study on properties and  applications of grand Lebesgue spaces, we refer  to  \cite{fior,  c}.

In the following theorem,  we obtain  the boundedness of the Hausdorff operator   in    $L^{p)}((0, 1), A_{\alpha, \beta})  $.

\begin{theorem}
		Let $1<p<\infty $  and   $\phi \in L ^1(0, \infty )$ be  such that $\supp \phi \subset [1, \infty).$ If $E(\phi, q)=\int_{1}^{\infty} \frac{\phi(t)}{t} t^{\frac{1}{q}} \;dt<\infty$ for some $q\in (0, p)$, then      $H_{\alpha, \beta, \phi}:L^{p)}((0, 1), A_{\alpha, \beta})  \rightarrow L^{p)}((0, 1), A_{\alpha, \beta})$ is a bounded operator and 
 	\begin{align*}
		\left\|H_{\alpha, \beta, \phi} f\right\|_{L^{p)}((0, 1), A_{\alpha, \beta})}&\leq 	 (A_{\alpha, \beta} (1))^2(p-1) \inf_{0<\sigma<p-1}       \sigma^{-\frac{1}{p-\sigma}} E(\phi, p-\sigma)   \|f\|_{L^{p)}((0, 1), A_{\alpha, \beta})}.
  	\end{align*}
\end{theorem}
\begin{proof}
  Let   us fix $\sigma \in(0, p-1)$.  Then
	\begin{align*}
		&\left\|H_{\alpha, \beta, \phi} f\right\|_{L^{p)}((0, 1), A_{\alpha, \beta})}\\&=\sup _{0<\varepsilon<p-1} \varepsilon^{\frac{1}{p-\varepsilon}}\left(\frac{1}{A_{\alpha, \beta}((0, 1))} \int_{0}^1|H_{\alpha, \beta, \phi} f(x)|^{p-\varepsilon} \;A_{\alpha, \beta}(x) d x\right)^{\frac{1}{p-\varepsilon}}\\
		&= \max \left\{\sup _{0<\varepsilon \leq \sigma}  \left(\frac{\varepsilon}{A_{\alpha, \beta}((0, 1))} \int_{0}^1|H_{\alpha, \beta, \phi} f(x)|^{p-\varepsilon} \;A_{\alpha, \beta}(x) dx\right)^{\frac{1}{p-\varepsilon}},\right.\\
		&\left.  \qquad \sup _{\sigma<\varepsilon<p-1}\left( \frac{\varepsilon}{A_{\alpha, \beta}((0, 1))}  \int_{0}^1|H_{\alpha, \beta, \phi} f(x)|^{p-\varepsilon} \;A_{\alpha, \beta}(x) d x\right)^{\frac{1}{p-\varepsilon}}\right\}.
	\end{align*}
Using Theorem \ref{eq13} and H\"older's inequality  for the conjugate exponents  $ \frac{p-\sigma}{p-\varepsilon}$ and $ \frac{p-\sigma}{\varepsilon-\sigma }$, we get
	\begin{align*}
			&\left\|H_{\alpha, \beta, \phi} f\right\|_{L^{p)}((0, 1), A_{\alpha, \beta})}\\
		&=  \max \left\{\sup _{0<\varepsilon \leq \sigma}  \left(\frac{\varepsilon}{A_{\alpha, \beta}((0, 1))} \int_{0}^1|H_{\alpha, \beta, \phi} f(x)|^{p-\varepsilon} \;A_{\alpha, \beta}(x) d x\right)^{\frac{1}{p-\varepsilon}}, \sup _{\sigma<\varepsilon<p-1} \left(\frac{\varepsilon}{A_{\alpha, \beta}((0, 1))}\right)^{\frac{1}{p-\varepsilon}} \right.\\
		&\left.  \qquad \times   \left(    \int_{0}^1|H_{\alpha, \beta, \phi} f(x)|^{p-\sigma} \;A_{\alpha, \beta}(x) d x\right)^{\frac{1}{p-\sigma}}
		\left(   \int_{0}^1\;A_{\alpha, \beta}(x) d x\right)^{\frac{\varepsilon-\sigma}{(p-\sigma)(p-\varepsilon)}}\right\}\\
			&\leq  \max \left\{\sup _{0<\varepsilon \leq \sigma}  \left(\frac{\varepsilon}{A_{\alpha, \beta}((0, 1))} \int_{0}^1|H_{\alpha, \beta, \phi} f(x)|^{p-\varepsilon} \;A_{\alpha, \beta}(x) d x\right)^{\frac{1}{p-\varepsilon}}, \sup _{\sigma<\varepsilon<p-1} \varepsilon^{\frac{1}{p-\varepsilon}} \right.\\
		&\left.  \qquad \times \left(\frac{\sigma}{A_{\alpha, \beta}((0, 1))}\right)^{-\frac{1}{p-\sigma}} \left( \frac{\sigma}{A_{\alpha, \beta}((0, 1))}  \int_{0}^1|H_{\alpha, \beta, \phi} f(x)|^{p-\sigma} \;A_{\alpha, \beta}(x) d x\right)^{\frac{1}{p-\sigma}}
	A_{\alpha, \beta} (1)^{\frac{p-\sigma-1}{p-\sigma}}\right\}\\
	&	\leq   \max \left\{1, (p-1) \left(\frac{\sigma}{A_{\alpha, \beta}((0, 1))}\right)^{-\frac{1}{p-\sigma}}	A_{\alpha, \beta} (1)^{\frac{p-\sigma-1}{p-\sigma}} \right\} \\&\qquad \times  \sup _{0<\varepsilon\leq \sigma}\left(\frac{\varepsilon}{A_{\alpha, \beta}((0, 1))}    \int_{0}^1|H_{\alpha, \beta, \phi} f(x)|^{p-\varepsilon} \;A_{\alpha, \beta}(x) d x\right)^{\frac{1}{p-\varepsilon}}\\
	&\leq   \max \left\{1,(p-1)\left(\frac{\sigma}{A_{\alpha, \beta}((0, 1))}\right)^{-\frac{1}{p-\sigma}}	A_{\alpha, \beta} (1)^{\frac{p-\sigma-1}{p-\sigma}} \right\}\\&\qquad\times   A_{\alpha, \beta}(1)^{1-\frac{1}{p}} \sup _{0<\varepsilon\leq \sigma} E(\phi, p-\varepsilon)   \left(\frac{\varepsilon}{A_{\alpha, \beta}((0, 1))}       \int_{0}^1|  f(x)|^{p-\varepsilon} \;A_{\alpha, \beta}(x) d x\right)^{\frac{1}{p-\varepsilon}}\\
	&\leq      (A_{\alpha, \beta} (1))^2 (p-1)\sigma^{-\frac{1}{p-\sigma}}	E(\phi, p-\sigma )   \|f\|_{L^{p)}((0, 1), A_{\alpha, \beta})}.
	\end{align*}
Now, taking the infimum over $ \sigma \in (0, p-1)$, we get 
		\begin{align*}
		\left\|H_{\alpha, \beta, \phi} f\right\|_{L^{p)}((0, 1), A_{\alpha, \beta})}&\leq 	 (A_{\alpha, \beta} (1))^2 (p-1) \inf_{0<\sigma<p-1}       \sigma^{-\frac{1}{p-\sigma}} E(\phi, p-\sigma)   \|f\|_{L^{p)}((0, 1), A_{\alpha, \beta})},
	\end{align*}
	and this completes the proof.
\end{proof}
Now, we  give  a necessary condition for the  $ L^{p)}((0, 1), A_{\alpha, \beta}) $-boundedness of the Hausdorff operator.
 \begin{theorem}
	Let $1<p<\infty $  and   $\phi \in L ^1(0, \infty )$. If   $H_{\alpha, \beta, \phi} :L^{p)}((0, 1), A_{\alpha, \beta})  \rightarrow L^{p)}((0, 1), A_{\alpha, \beta})$ is a bounded operator,  then $$\left\|H_{\alpha, \beta, \phi} \right\|_{L^{p)}((0, 1), A_{\alpha, \beta})  \rightarrow L^{p)}((0, 1), A_{\alpha, \beta}) }\geq \frac{1}{A_{\alpha, \beta}(1)^{1-\frac{1}{p}}}  E\left(\phi, p\right),$$
	where $E\left(\phi, p\right)$ as  in Theorem \ref{eq13}.
\end{theorem}
 \begin{proof}
 	For a fixed $\delta$ with  $ \delta< \min ( \frac{1}{p}, 1-\frac{1}{p})$, we define  the  function 
 	$$
 	f_\delta(x)= x^{\delta-\frac{1}{p} } A_{\alpha, \beta}(x)^{-\frac{1}{p}}, \quad x\in  (0, 1).
 	$$
 	Then 
 	\begin{align*}
 		\|f_\delta\|_{L^{p)}((0, 1), A_{\alpha, \beta}) }&=\sup _{0<\varepsilon<p-1} \varepsilon^{\frac{1}{p-\varepsilon}}\left(\frac{1}{A_{\alpha, \beta}((0, 1))} \int_{0}^1| f_\delta(x) |^{p-\varepsilon} \;A_{\alpha, \beta}(x) d x\right)^{\frac{1}{p-\varepsilon}}\\
 		&=\sup _{0<\varepsilon<p-1} \varepsilon^{\frac{1}{p-\varepsilon}}\left(\frac{1}{A_{\alpha, \beta}((0, 1))} \int_{0}^1x^{(\delta-\frac{1}{p}) (p-\varepsilon)} A_{\alpha, \beta}(x)^{-\frac{p-\varepsilon}{p}}\;A_{\alpha, \beta}(x) d x\right)^{\frac{1}{p-\varepsilon}}\\
 		&\leq A_{\alpha, \beta}(1)  \sup _{0<\varepsilon<p-1} \varepsilon^{\frac{1}{p-\varepsilon}}\left( \int_{0}^1x^{(\delta-\frac{1}{p}) (p-\varepsilon)}   d x\right)^{\frac{1}{p-\varepsilon}}\\
 		&=A_{\alpha, \beta}(1) \sup _{0<\varepsilon<p-1}\left(\frac{\varepsilon}{\left(\delta-\frac{1}{p}\right)(p-\varepsilon)+1}\right)^{\frac{1}{p-\varepsilon}}\\
 		&\leq A_{\alpha, \beta}(1) \frac{p-1}{\delta p}.
 		\end{align*} 
 	Also,  from the relation (\ref{eq14}),  for any $x\in (0, 1)$, we have
 	\begin{align*}
 		H_{\alpha, \beta, \phi}f_\delta(x)\geq  \frac{1}{A_{\alpha, \beta}(1)^{1-\frac{1}{p}}} E\left(\phi, \frac{p}{1-\delta p}\right)f_\delta(x) .
 	\end{align*}
 	Therefore,
 	\begin{align*} 
 		\left\|H_{\alpha, \beta, \phi} f_\delta\right\|_{L^{p)}((0, 1), A_{\alpha, \beta}) }
 		&\geq \frac{1}{A_{\alpha, \beta}(1)^{1-\frac{1}{p}}}  E\left(\phi, \frac{p}{1-\delta p}\right)\|f_\delta\|_{L^{p)}((0, 1), A_{\alpha, \beta}) },
 	\end{align*} 
 	and thus
 	\begin{align*} 
 		\left\|H_{\alpha, \beta, \phi} \right\|_{L^{p)}((0, 1), A_{\alpha, \beta}) \to L^{p)}((0, 1), A_{\alpha, \beta}) } \geq  \frac{1}{A_{\alpha, \beta}(1)^{1-\frac{1}{p}}}  E\left(\phi, \frac{p}{1-\delta p}\right).
 	\end{align*}
Now, taking the limit   $\delta \to 0$, we obtain 
 	\begin{align*} 
 		\left\|H_{\alpha, \beta, \phi} \right\|_{L^{p)}((0, 1), A_{\alpha, \beta})  \to L^{p)}((0, 1), A_{\alpha, \beta}) } \geq  \frac{1}{A_{\alpha, \beta}(1)^{1-\frac{1}{p}}}  E\left(\phi, p\right),
 	\end{align*}
 	and this completes the proof of the theorem.
 \end{proof}
 \subsection{Boundedness of the Hausdorff operator in quasi-Banach spaces}
 In this subsection,  we   study the   boundedness of the Hausdorff operator in the quasi-Banach space associated with the Opdam--Cherednik transform.
  First, we  recall the following lemma.
 \begin{lemma}\cite{weight}\label{eq17}
 	Let $0<s<1, \;-\infty<a<b \leq  \infty$ and $h$ be a non-negative and non-increasing function defined on the interval $(a, b)$, then
 	$$
 	\left(\int_{a}^{b} h(t) \;d t\right)^{s} \leq s \int_{a}^{b} h^{s}(t)(t-a)^{s-1} \;d t.
 	$$
 \end{lemma}
 For $\phi \in L ^1(0, \infty )$, let  $\mathcal{M}_{\phi}$ be the class of measurable 
 functions $f: \R \rightarrow \R$
  such that $t \mapsto \frac{\phi(t)}{t} f\left(\frac{x}{t}\right) \frac{A_{\alpha, \beta}(\frac{x}{t})}{A_{\alpha, \beta}(x)}$ is non-increasing. We define  two quantities $	B_{\sup} $ and $B_{\inf }$   as
 $$
 \begin{aligned}
 	B_{\sup} &:=\left(  \int_{0}^{\infty}  \phi(t)^p  \left(\sup_{u\in \R}\frac{A_{\alpha, \beta}(u)}{A_{\alpha, \beta}(tu)} \right)^{p-1}  dt\right)^{\frac{1}{p}},\\
 	B_{\inf } &:=\left (  \int_{0}^{\infty}  \phi(t)^p     \left(\inf_{u\in \R} \frac{A_{\alpha, \beta}(u)}{A_{\alpha, \beta}(tu)}\right) ^{p-1} dt  \right)^\frac{1}{p}.
 \end{aligned}
 $$
 In the following, we prove   the  $L^p(\R, A_{\alpha, \beta})$-boundedness of the Hausdorff operator  in the  quasi-Banach space $L^p(\R, A_{\alpha, \beta})\cap \mathcal{M_\phi}$.
\begin{theorem} \label{eq23}
	Let $0<p<1 $  and   $\phi \in L ^1(0, \infty ).$  If $B_{\text {sup }}<\infty$, then  for any  $f \in L^p(\R, A_{\alpha, \beta})\cap \mathcal{M_\phi}$, the operator   $H_{\alpha, \beta, \phi}: L^p(\R, A_{\alpha, \beta})  \rightarrow L^p(\R, A_{\alpha, \beta})$ is   bounded   and  
	$$
	\left\|H_{\alpha, \beta, \phi} f\right\|_{L^p(\R, A_{\alpha, \beta})} \leq p^{\frac{1}{p}} B_{\sup}\|f\|_{L^p (\R, A_{\alpha, \beta})}.
	$$  
\end{theorem}
\begin{proof}
	For any  $f \in L^p(\R, A_{\alpha, \beta})\cap \mathcal{M_\phi}$, using   Lemma  \ref{eq17} with $a=0, b=\infty, s=p$, and  Fubini's theorem,  we get  
	\begin{align*}
		\left\|H_{\alpha, \beta, \phi} f\right\|_{L^p(\R, A_{\alpha, \beta})}
		&=\left( \int_{\R} 	|H_{\alpha, \beta, \phi} f(x)|^p\; A_{\alpha, \beta}(x)dx\right)^{\frac{1}{p}}\\
		&=\left( \int_{\R} 	\left |\int_{0}^{\infty} \frac{\phi(t)}{t} f\left(\frac{x}{t}\right) \frac{A_{\alpha, \beta}(\frac{x}{t})}{A_{\alpha, \beta}(x)}\; dt\right|^p A_{\alpha, \beta}(x)dx\right)^{\frac{1}{p}}\\
	&\leq \left( \int_{\R} p \left(   \int_{0}^{\infty}  \frac{\phi(t)^p}{t^p} t^{p-1}	  \left |f\left(\frac{x}{t}\right)\right|^p  \left( \frac{A_{\alpha, \beta}(\frac{x}{t})}{A_{\alpha, \beta}(x)}\right)^p dt \right) A_{\alpha, \beta}(x)dx\right)^{\frac{1}{p}}\\
		&= \left(  p\int_{0}^{\infty}  \frac{\phi(t)^p}{t }  \left( \int_{\R}  \left |f\left(\frac{x}{t}\right)\right|^p  \left( \frac{A_{\alpha, \beta}(\frac{x}{t})}{A_{\alpha, \beta}(x)}\right)^p A_{\alpha, \beta}(x)dx\right) dt \right) ^{\frac{1}{p}}.
	\end{align*} 
Using the   change of variable  $x  \mapsto  u = \frac{x}{t}$ in the second integral,  we obtain  
	\begin{align*}
		\left\|H_{\alpha, \beta, \phi} f\right\|_{L^p(\R, A_{\alpha, \beta})}
		&\leq   \left(  p\int_{0}^{\infty}   \phi(t)^p  \left( \int_{\R}  \left |f\left(u\right)\right|^p  \left( \frac{A_{\alpha, \beta}(u)}{A_{\alpha, \beta}(tu)}\right)^p A_{\alpha, \beta}(tu) du\right) dt \right) ^{\frac{1}{p}}\\
		&\leq  \left( p\int_{0}^{\infty}  \phi(t)^p  \left(\sup_{u\in \R}\frac{A_{\alpha, \beta}(u)}{A_{\alpha, \beta}(tu)} \right)^{p-1}  dt\right)^{\frac{1}{p}} \left(  \int_{\R}	 |f(u)|^p  A_{\alpha, \beta}(u) \; du \right)^{\frac{1}{p}}\\
		&= p^{\frac{1}{p}} B_{\sup}\|f\|_{L^p (\R, A_{\alpha, \beta})}.
	\end{align*} 
This completes the proof. 
\end{proof}
Next, we  provide  a necessary condition for  the $L^p(\R, A_{\alpha, \beta})$-boundedness of the Hausdorff  operator in the  quasi-Banach space $L^p(\R, A_{\alpha, \beta})\cap \mathcal{M_\phi}$.
\begin{theorem}  \label{eq24}
Let $0<p<1$, $\phi \in L ^1(0, \infty )$ and $B_{\inf} > 0$. If $H_{\alpha, \beta, \phi}: L^p(\R, A_{\alpha, \beta})  \rightarrow L^p(\R, A_{\alpha, \beta})$ is a bounded operator, then  
$$
\left\|H_{\alpha, \beta, \phi} \right\|_{L^p (\R, A_{\alpha, \beta})\to L^p (\R, A_{\alpha, \beta})} \geq p^{\frac{1}{p}}B_{\inf}.$$
\end{theorem}
\begin{proof}
	Suppose that   $H_{\alpha, \beta, \phi}: L^p(\R, A_{\alpha, \beta})  \rightarrow L^p(\R, A_{\alpha, \beta})$ is a bounded operator. We consider the  function 
	$$
	f_0(x)= x^{-\frac{1}{p}-1} A_{\alpha, \beta}(x)^{-\frac{1}{p}}\chi_{(1, \infty )}(x).
	$$
	Then, we have 
	$$\|f_0\|_{L^p (\R, A_{\alpha, \beta})}=\left( \int_{\R} 	|f_0(x)|^p\; A_{\alpha, \beta}(x)dx\right)^{\frac{1}{p}}=\left( \int_{1}^\infty 	x^{-1-p}\; dx\right)^{\frac{1}{p}}=\frac{1}{ p^{\frac{1}{p}}}.$$
Also, using  the reverse Minkowski inequality, we get 
	\begin{align*} 
		\left\|H_{\alpha, \beta, \phi} f_0\right\|_{L^p(\R, A_{\alpha, \beta})}
		&=\left( \int_{\R} 	|H_{\alpha, \beta, \phi} f_0(x)|^p\; A_{\alpha, \beta}(x)dx\right)^{\frac{1}{p}}\\\nonumber
		&\geq \left( \int_{0}^\infty  	\left |\int_{0}^{\infty} \frac{\phi(\frac{x}{t})}{t} f_0 (t ) \frac{A_{\alpha, \beta}(t)}{A_{\alpha, \beta}(x)}\; dt \right|^p A_{\alpha, \beta}(x)dx\right)^{\frac{1}{p}}\\
		&= \left( \int_{0}^\infty  	\left ( \int_{1}^{\infty} \frac{\phi(\frac{x}{t})}{t} t^{-\frac{1}{p}-1}  \frac{A_{\alpha, \beta}(t)^{1-\frac{1}{p}}}{A_{\alpha, \beta}(x)}\; dt \right)^p A_{\alpha, \beta}(x)dx\right)^{\frac{1}{p}}\\
		&\geq  \int_{1}^\infty  \frac{1}{t^2}	\left (  \int_{0}^{\infty} \frac{\phi(\frac{x}{t})^p}{t }    \left( \frac{A_{\alpha, \beta}(t)}{A_{\alpha, \beta}(x)}\right) ^{p-1} dx  \right)^\frac{1}{p}   dt.
		\end{align*} 
Using the   change of variable  $x  \mapsto  u = \frac{x}{t}$ in the second integral,  we obtain  
		\begin{align*} 
		\left\|H_{\alpha, \beta, \phi} f_0\right\|_{L^p(\R, A_{\alpha, \beta})}	
		&\geq\int_{1}^\infty  \frac{1}{t^2}	\left ( \int_{0}^{\infty}  \phi(u)^p     \left( \frac{A_{\alpha, \beta}(t)}{A_{\alpha, \beta}(ut)}\right) ^{p-1} du  \right)^\frac{1}{p}  dt\\
	&\geq \left( \int_{1}^\infty  \frac{1}{t^2}	dt\right)\left ( \int_{0}^{\infty}  \phi(u)^p     \left(\inf_{t\in \R} \frac{A_{\alpha, \beta}(t)}{A_{\alpha, \beta}(ut)}\right) ^{p-1} du  \right)^\frac{1}{p}  \\
		&= \left (  \int_{0}^{\infty}  \phi(u)^p     \left(\inf_{t\in \R} \frac{A_{\alpha, \beta}(t)}{A_{\alpha, \beta}(ut)}\right) ^{p-1} du  \right)^\frac{1}{p}.
	\end{align*} 
	Thus,
	\begin{align*} 
		\left\|H_{\alpha, \beta, \phi} \right\|_{L^p (\R, A_{\alpha, \beta})\to L^p (\R, A_{\alpha, \beta})} \geq p^{\frac{1}{p}}B_{\inf}.
	\end{align*}
\end{proof}
From Theorems \ref{eq23} and \ref{eq24}, in the following corollary,  we obtain a  characterization for the boundedness of the  Hausdorff operator $H_{\alpha, \beta, \phi}: L^p(\R, A_{\alpha, \beta}) \rightarrow L^p(\R, A_{\alpha, \beta})$.
\begin{corollary}
	Let $0<p<1$ and 
	$$
	\sup_{u\in \R}\frac{A_{\alpha, \beta}(tu)}{A_{\alpha, \beta}(u)} \leq  D \inf_{u\in \R}\frac{A_{\alpha, \beta}(tu)}{A_{\alpha, \beta}(u)}, \quad t>0,
	$$
	for some positive constant $D .$ Then,  the operator $H_{\alpha, \beta, \phi}:L^p(\R, A_{\alpha, \beta})  \rightarrow L^p(\R, A_{\alpha, \beta})$ is  bounded    if and only if 
	$0< B_{\sup} <\infty.$  Moreover,  the
	following estimates hold
	$$
	\frac{p^{\frac{1}{p}}}{D^{\frac{1}{p}-1}} B_{\sup} \leq \left\|H_{\alpha, \beta, \phi} \right\|_{L^p (\R, A_{\alpha, \beta})\to L^p (\R, A_{\alpha, \beta})} \leq  p^{\frac{1}{p}} B_{\sup}.
	$$
\end{corollary}

\section*{Acknowledgments}
The first author gratefully acknowledges the support provided by IIT Guwahati, Government of India. 
The second author   is deeply indebted to Prof. Nir Lev for several fruitful discussions and generous comments. The authors wish to thank the anonymous referee for valuable comments and suggestions that helped to improve the quality of the paper.

\section*{Data availability statements}	The authors confirm that the data supporting the findings of this study are available within the article and its supplementary materials.

\section*{Declarations} {\bf Conflict of interest} The authors declare that there is no conflict of interest regarding the publication of this article.

\end{document}